\documentclass[leqno,twoside,11pt]{amsart}
\usepackage{latexsym,esint,url, amssymb, yfonts}
\usepackage{cancel, soul}
\usepackage[colorinlistoftodos,prependcaption,textsize=tiny]{todonotes}

\usepackage{graphicx,color}
\usepackage{bm}
\usepackage[normalem]{ulem}

\usepackage{mathtools} 

\DeclareRobustCommand{\rchi}{{\mathpalette\irchi\relax}}
\newcommand{\irchi}[2]{\raisebox{\depth}{$#1\chi$}} 

\makeatletter
\@namedef{subjclassname@2020}{\textup{2020} Mathematics Subject Classification}
\makeatother

\theoremstyle{plain}

\newtheorem{theorem}{Theorem}
\newtheorem{lemma}{Lemma}[section]
\newtheorem{corollary}[theorem]{Corollary}
\newtheorem{proposition}[lemma]{Proposition}

\newtheorem{remark}[lemma]{Remark}
\newtheorem{definition}[lemma]{Definition}

\newtheorem{conjecture}[theorem]{Conjecture}

\numberwithin{figure}{section}

 \newcommand{\R}{\mathbb{R}} 

\newcommand{\bee}{\begin{equation}}
\newcommand{\eee}{\end{equation}} 
\newcommand{\bees}{\begin{equation*}}
\newcommand{\eees}{\end{equation*}} 
\newcommand{\ds}{\displaystyle} 

\numberwithin{equation}{section}
\numberwithin{figure}{section}

\def\bees{\begin{equation*}}
\def\eees{\end{equation*}}
\def\bee{\begin{equation}}
\def\eee{\end{equation}}

\def\ds{\displaystyle}
\def\Det{{\mathcal{D}et}}
\def\cc{{\rm curl}~ {\rm curl}\,} 
 
\def\R{{\mathbb R}}
\def\e{\varepsilon} 

\definecolor{Green}{rgb}{0, 0.65,0}

\begin{document}

\title[Convexity of weakly regular surfaces]
{Convexity of weakly regular surfaces of distributional nonnegative intrinsic curvature} 
\author{Mohammad Reza Pakzad}

\address[Mohammad Reza Pakzad]{Laboratoire IMATH, Universit\'e de Toulon, 
CS 60584, 83041 TOULON CEDEX 9,  France}
 \email{pakzad@univ-tln.fr}

\subjclass[2020]{53A05, 53C24, 53C21, 35J96}
\keywords{rigidity, surfaces of nonnegative curvature, differential geometry at low regularity}
 
\date{}
 
\begin{abstract} We prove that the image of an isometric embedding into $\R^3$ of a  two dimensionnal complete  Riemannian manifold $(\Sigma, g)$ without boundary  is a convex surface, provided that, first, both the embedding and the metric $g$ enjoy a $C^{1,\alpha}$ regularity for some $\alpha>2/3$, and second, the distributional Gaussian curvature  of $g$ is nonnegative and nonzero.  The analysis must pass through some key observations regarding solutions to the very weak Monge-Amp\`ere equation.  
 \end{abstract}

\maketitle
\section{Introduction.}\label{intro}

\subsection{Background}
It  is a well-known fact of differential geometry that any $C^2$-smooth complete surface  in $\R^3$, whose Riemannian metric has non-negative and non-zero intrinsic Gaussian curvature,  is a convex surface, i.e.\@ it is the boundary of a convex region. This statement is one of the consequences of the fact that, by  Theorema Egregium, the intrinsic and extrinsic notions of curvatures coincide for a $C^2$ surface.  

One could directly work with the notion of extrinsic curvature to obtain similar results. Pogorelov generalized the above statement on convexity of surfaces to the case of $C^1$ surfaces with bounded extrinsic curvature, see \cite[Theorem 2, p.\@ 615]{Po73}. On the other hand, by the celebrated results of Nash and Kuiper \cite{Nash2, kuiper}, there exist $C^1$ non-convex surfaces which are isometrically equivalent with the unit sphere. The Riemannian metric of these surfaces is that of the sphere whose intrinsic  curvature is positive and constant; yet they are not surfaces of bounded extrinsic curvature. The following question can hence be put forward: {\bf Under which regularity assumptions on a surface and its metric, the intrinsic and extrinsic notions of the Gaussian curvature coincide?} 

The above question is a variant of the many problems on the flexibility vs.\@ rigidity dichotomy regarding the solutions to nonlinear PDEs; (see \cite{DS-survey} for a background survey on this aspect of our problem). On one hand, it can be shown that the identity of the two curvatures fails to be true in general for surfaces of $C^{1,\alpha}$ regularity with $\alpha<1/5$ \cite{1/5,  CSz19}.   On the other hand, when $\alpha>2/3$ and $\beta>0$, positive answers to the question was provided in \cite{bori1} and in \cite{CDS} (see also \cite{CDS-errata}) for $C^{1,\alpha}$-isometric embeddings of 2d complete surfaces, whose Riemannian metrics are $C^{2,\beta}$-regular and of  positive Gaussian curvature.  
 
In this article, we  generalize  the main results of \cite{bori1, CDS} as follows: For the same  exponent regime $\alpha>2/3$, we assume only a  $C^{1,\alpha}$ regularity for the Riemannian metric, and we relax the positivity condition of the curvature to the condition that the now distributionally defined intrinsic curvature is non-zero and nonnegative. The main ingredient of the proof is to show that provided  a Riemannian metric $g\in C^{1,\alpha}$ on a domain $\Omega \subset \R^2$, whose distributional Gaussian curvature  is nonnegative, and any isometric embedding  $u\in C^{1,\alpha}$ of $(\Omega, g)$ into $\R^3$, the image $u(\Omega)$ is of bounded extrinsic curvature. Our analysis uses the ideas of Conti, De Lellis, and Sz\'ekelyhidi  in \cite{CDS, CDS-errata} but cannot directly follow their methodology: The best  conceivable result from their approach, if only the nonnegativity of the curvature is assumed, is that the degree of the Gauss map $\vec n$ at any point in its image is nonnegative, which  is not enough to conclude with a bound on the total variation of  the areas of  images of $\vec n$ (a necessary step in bounding the extrinsic curvature). 
\medskip 

In order to pass to the analysis of  the nonnegative curvature case, the image surface close to a point is considered as the graph of a function $v$. Writing the curvature in terms of $v$ leads to a version of the Monge-Amp\`ere equation with nonnegative data. A perturbation argument, originally due to Kirchheim \cite{Ki01}, shows that  if  $v$ satisfies  $\det (\nabla^2 v) \ge 0$, then $\nabla v$ can be uniformly approximated with deformations having positive Jacobian determinants, yielding that the positivity of the topological degree of $\nabla v$ on its image points. This latter conclusion is the key for proving the bounded extrinsic property of the graph of $v$. The main technical difficulty is to  apply this approach in  a weak regularity setting, specially in a context where nonlinear expressions involving  distributions in negative order function spaces enter into the picture.

 \medskip
 
  To put  our results in a more general context, our results must be compared with the statements in \cite{Nash2, kuiper, Po56, bori1, bori2, Po73, CDS, CDS-errata, HoVe, 1/5, CSz19, DP20, LPS21} on isometric immersions,  in \cite{Ki01, Pak, Sve, LMP17} on the rigidity of Sobolev solutions to the Monge-Amp\`ere equation,  and in \cite{BN11, Kor, Kor2, GO20, LS20, LPS21} on geometric or topological properties of  weakly regular deformations.

\subsection{Main results}
 
 We adapt the terminology of \cite{GMS}.   We say $f\ge h$  for two distributions $f,h\in \mathcal{D}'(\Omega)$, when 
 $$
\forall \varphi \in C^\infty_c(\Omega) \,\, \varphi \ge 0 \implies f [\varphi] \ge h  [\varphi]. 
$$   We remark that by \cite[Theorem 1.4.2(ii)]{GMS} any nonnegative distribution $f \in \mathcal D'(\Omega)$ can be extended as bounded linear operator on $C_c(\Omega)$ and hence induces a Radon measure  $\mu_f \in {\mathcal M}^+(\Omega)$, satisfying
$$
\forall \varphi \in C^\infty_c(\Omega)\,\, \int_\Omega \varphi\,{\rm d}\mu_f = f [\varphi].
$$
 Conversely, any Radon measure  induces a nonnegative distribution. Throughout the paper, we will hence use the notions of Radon measures and nonnegative distributions interchangeably. For $\mu, \nu \in \mathcal{M}(\Omega)$ we say $\mu\ge \nu$   if and only if  for all Borel sets $A \subset \Omega\,\, \, \mu(A) \ge  \nu(A)$. This is an equivalent condition to $f \ge h$ as distributions, when $\mu$ and $\nu$ are respectively induced by $f$ and  $h$. 

\medskip
Here  we also  motivate and provide the definition of the distributional Gaussian curvature -which generalizes the standard notion to Riemannian 
metrics of lower regularity.  Let $\Omega \subset \R^2$ be any open set. We recall that the Christoffel symbols associated with  a Riemannian metric $g =[g_{ij}] \in C^2(\Omega, \R^{2\times 2}_{sym, pos})$       are given by
$$
\Gamma^i_{jk} (g) = \frac 12 g^{im} (\partial_k g_{jm} +   \partial_j g_{km} - \partial_m g_{jk}),
$$ 
where the Einstein summation convention is used. In view of the formula for the $(0,4)$-Riemann curvature tensor \cite[Equation (2.1.2)]{hanhong}
$$
R_{iljk} = g_{lm} (\partial_k \Gamma^m_{ij}  -\partial_j \Gamma^m_{ik}   + \Gamma^m_{ks} \Gamma^s_{ij} - \Gamma^m_{js} \Gamma^s_{ik}),  
$$ we have:
$$
\begin{aligned}
 R_{1212}(g) & = g_{1m}(\partial_1 \Gamma^m_{22} - \partial_2 \Gamma^m_{21} + \Gamma^m_{1s} \Gamma^{s}_{22}  
- \Gamma^m_{2s} \Gamma^{s}_{21})\\ & =
\partial_1 (g_{1m} \Gamma^m_{22} ) - \partial_2 ( g_{1m} \Gamma^m_{21})  -\partial_1 g_{1m} \Gamma^m_{22} 
+ \partial_2 g_{1m} \Gamma^m_{21} + g_{1m}(\Gamma^m_{1s} \Gamma^{s}_{22}  
- \Gamma^m_{2s} \Gamma^{s}_{21}) 
\\&  =  2 \partial_{12} g_{12} - \partial_{11} g_{22} - \partial_{22} g_{11}   -\partial_1 g_{1m} \Gamma^m_{22} 
+ \partial_2 g_{1m} \Gamma^m_{21}  + g_{1m}(\Gamma^m_{1s} \Gamma^{s}_{22}  
- \Gamma^m_{2s} \Gamma^{s}_{21})  \\ & = - {\rm curl}\, {\rm curl}\,\, g    -\partial_1 g_{1m} \Gamma^m_{22} 
+ \partial_2 g_{1m} \Gamma^m_{21}  + g_{1m}(\Gamma^m_{1s} \Gamma^{s}_{22}  
- \Gamma^m_{2s} \Gamma^{s}_{21}).
 \end{aligned} 
$$  Moreover, by the Gauss equation \cite[Equation (2.1.7)]{hanhong}, we have for the Gaussian curvature of $g$:
$$
\kappa(g):= \frac{R_{1212}}{\det g}. 
$$   The above calculations motivate the following definition:
\begin{definition}
Assume $U \subset \R^2$ is an open set, and $g\in  C^1 (U, \R^{2\times 2}_{sym, pos}) $. 
Let $$
L(g):= \frac{1}{\det g} \Big (-\partial_1 g_{1m} \Gamma^m_{22} 
+ \partial_2 g_{1m} \Gamma^m_{21}  + g_{1m}(\Gamma^m_{1s} \Gamma^{s}_{22}  
- \Gamma^m_{2s} \Gamma^{s}_{21}) \Big ) \in C(U) . 
$$
The distributional Gaussian curvature $\kappa_g$  of $g$ is defined  by
$$
\forall \varphi \in C_c^\infty(U) \quad \kappa_g [\varphi]:=  \int_U \frac{1}{\det g} ({\rm curl}\,\, g) \cdot  \nabla^\perp \varphi +  
({\rm curl}\,\, g)\cdot  \nabla^\perp   (\frac{1}{\det g}) \varphi + L(g) \varphi,
$$  which can be extended to a bounded operator on $W^{1,1}_0(U)$ if $g\in C^1(\overline U)$.
 \end{definition}     
       
 We finally need to provide the following definition on surfaces of bounded extrinsic curvature, which follows
 \cite[p.\@ 590]{Po73}.
       
\begin{definition}\label{bxc}
Let $S\subset \R^3$ be a $C^1$ surface.  We say that $S$ has the  
bounded extrinsic curvature property whenever, the set function ${\mathcal H}^2 ({\vec n}(F))$, defined for $F\subset S$ and  induced by the Gauss map ${\vec n}:S\to {\mathbb S}^2$,  is of bounded total variation over $S$, i.e., there exists a constant $C=C(S)<\infty$ such that for any finite collection $\{F_i\}_{i=1}^{k}$ of pairwise disjoint closed subsets of $S$
$$
\sum_{i=1}^k  {\mathcal H}^2 ({\vec n}(F_i)) < C.
$$  \end{definition}  
 
If a surface $S$ has the  bounded extrinsic curvature property, the {\em absolute curvature} 
of an open set $O$ in $S$  is defined to be the supremum of  $\sum_{i=1}^k {\mathcal H}^2({\vec n}(F_i))$, over all
finite collections $\{F_i\}_{i=1}^k$ of  closed
subsets of $O$ which are pairwise disjoint.  The absolute curvature of any subset $A$
of $S$ is defined as the infimum of its absolute curvatures on open
subsets containing $A$.  By \cite[p.\@ 590, Theorem 2]{Po73} (see also \cite[p.\@ 573, Theorem 1]{Po73}), the absolute curvature is a $\sigma$-additive set function on the $\sigma$-algebra of Borel subsets of $S$.

\medskip

For a surface of  bounded extrinsic curvature $S$, we will refer to  \cite[p.\@ 591]{Po73} for the  definitions of  {\em regular}, {\em elliptic}, {\em parabolic},  {\em hyperbolic}, and   {\em flat} points $p\in S$. 
In particular a point $p \in S$ is defined to be a 
{\em regular} point if the tangent plane at $p$ is not parallel to any other tangent plane in 
a  sufficiently small neighbourhood \begin{equation}\label{reg_def}
\exists  O\subset S \,\, \mbox{open neighborhood of} \,\,p  \quad \forall q\in O \setminus \{p\}  \quad {\vec n}(q)\neq {\vec n}(p).
\end{equation} For every set $A\subset S$,  the
 {\em positive (resp.\@ negative) curvature} of $A$ consists in the absolute
curvature of the subset of $S$ formed by elliptic (resp.\@ hyperbolic) points of $A$. The {\em total curvature} of $A$ is defined to be the
difference between its positive and negative curvatures, and  
$S$ is a {\em surface of nonnegative curvature} 
if the negative curvature of $S$ equals zero \cite[p.\@ 611]{Po73}.

\medskip


 Our main result is the following.
 
  \begin{theorem}\label{iso-bou-cur}
Assume that  $\Omega\subset \R^2$ is an open set and that $\alpha> 2/3$. Let $g$ be a $C^{1,\alpha}$ Riemannian metric on $\Omega$ for which the distributional Gaussian curvature is nonnegative (and hence is  a Radon measure over $\Omega$). If $u\in C^{1, \alpha}((\Omega, g), \R^3)$ is an isometric embedding, $S=u(\Omega)$, and $S'$ is a surface compactly contained in $S$, then  $S'$  has the bounded extrinsic curvature property, and  is of nonnegative curvature.   If $\kappa_g$ is, moreover, nonzero, then there exists $S'$ as above such that the positive curvature of $S'$ is non-zero.   
  \end{theorem}
  
\begin{remark}
One can potentially replace definition \ref{bxc} with a local variant: i.e.\@ to define surfaces of locally bounded extrinsic curvature  where the finiteness  of the total variation is only assumed to hold true on relatively compact open subsets of $S$. In that case a surface of locally bounded extrinsic curvature of nonnegative or nonzero curvature would be also well-defined, and  under the assumptions of Theorem \ref{iso-bou-cur}, one can show that $S=u(\Omega)$ is a surface of locally bounded extrinsic curvature,  of nonnegative curvature, which is of nonzero curvature provided  $\kappa_g$ is nonzero. We prefer to stay with the terminology of \cite{Po73} to avoid any confusion, even though the local variant is implicitly assumed in e.g.\@ \cite[Theorem 2, p.\@ 615]{Po73}.
 \end{remark}         

As we shall see in Proposition \ref{chart-cur},  the distributional Gaussian curvature is invariant under a natural change of coordinate 
formula for smooth enough changes of coordinates, establishing that its definition can be  extended to $C^1$ metrics over smooth  manifolds of two dimensions. The following corollary  of  Theorem \ref{iso-bou-cur} on surfaces of nonnegative distributional curvature is immediately obtained by \cite[Theorem 2, p.\@ 615]{Po73}, generalizing the results of \cite{bori1}, 
and  \cite[Corollary B]{CDS-errata}. Note that an examination of the proof of \cite[Theorem 2, p.\@ 615]{Po73} shows that it uses the bounded extrinsic curvature property assumption only on  bounded sections of  the possibly unbounded surface. We leave the verification of this detail to the reader. As a consequence, Theorem \ref{iso-bou-cur} can be applied to reach the following conclusion.     
  
  \begin{corollary} \label{cor-convex} 
 Let $\alpha> 2/3$, and $\Sigma$ be a  complete  2-dimensional  $C^2$ manifold with no boundary. Assume moreover that $g$ is a $C^{1,
 \alpha}$ Riemannian metric on $\Sigma$ for which the distributional Gaussian curvature is nonnegative and nonzero.  If $u\in C^{1, \alpha}((\Sigma, g), \R^3)$ is an isometric embedding, then, $u(\Sigma)$ is either a closed convex surface, or an unbounded  convex surface with no boundary. 
  \end{corollary}

 \begin{remark}\label{onefifth}
 The above  conclusions are not true if $\alpha<1/5$, even for smooth metrics with positive curvature \cite{1/5, CSz19}. The  question remains open for the range $1/5\le \alpha \le 2/3$. Gromov \cite[Section 3.5.5.C, Open Problem 34-36]{gromov2} conjectures the critical exponent to be $\alpha=1/2$, see also \cite{GO20}.
\end{remark} 
   
To achieve the above goals, we will have to study the very weak  Monge-Amp\`ere equation in  an open domain $\Omega \subset\R^2$: 
\bee \label{MA}
\Det {\mathcal D}^2 v  := -\frac 12 \cc  (\nabla v \otimes \nabla v) =  f \in \mathcal{D}'(\Omega) 
\eee  See \cite{Iwa, FM} for an introduction to the very weak Hessian determinant.  We will  in particular consider the $C^{1,\alpha}$-solutions\footnote{$~$Note that  the solutions are not assumed to be convex.}\@ $v$ to \eqref{MA} for which the given distribution $f\ge 0$ under the assumption that $\alpha> 2/3$.  Let us recall that, on the other hand, $v$ is called an Alexandrov solution to  $\det\nabla^2v = \mu$ on $\Omega$ when, for any convex subset $U\subset \Omega$,  $v$ is convex  on $U$, $\mu_v(A) = |\partial v(A)|$ defines a Borel measure on $U$, and  $\mu_v = \mu$. In this line another minor consequence of our analysis is the following statement.

\begin{theorem}\label{full-plane}
Assume that $2/3 < \alpha <1$ and that $f\in \mathcal D'(\R^2)$ is a nonnegative nonzero distribution.  If  $v\in C^{1,\alpha}(\R^2)$  is a solution to $\Det{\mathcal D}^2 v = f$ in $\mathcal{D'}(\R^2)$, then modulo a sign  the function $v$ is convex over $\R^2$ and  is an Alexandrov solution  to  
$\det\nabla^2v = \mu_f$ on $\R^2$.   \end{theorem}

  \begin{remark} 
  Note that since $v\in C^1$, we have $|\partial v(A)| =|\nabla v(A)|$ for all Borel sets $A \subset U$.  \end{remark}
   
  \begin{remark} A similar result for solutions on a general domain $\Omega \subset \R^2$ was announced by Lewicka and the author in \cite[Theorem 1.4]{LP17}. The proof was supposed to appear in a forthcoming paper, but a gap was found. The obstacle is not overcome to this day and the problem for arbitrary  domains remains open. 
  \end{remark}
      
 As it is suggested in Remark \ref{onefifth}, the main challenging  problem  is to extend the results of this paper  to the case $\alpha>1/2$.  Apart from partial results regarding isometric extensions \cite{DI20, CI20} for this regime, the problem is still largely open and cannot be answered using the current methods.  Here, we  conjecture that slight improvements are possible  regarding  the $2/3$ regime, in the line of analysis in \cite{LS20, LPS21}.

\begin{conjecture}
The above results can be generalized with some extra technical maneuvers to Besov regularity regime  $B^{1+s}_{2/s,\infty}$, for $s>2/3$.
\end{conjecture}

The paper is organized as follows: In Section \ref{MA-sec}, we will recall some known statements about the very weak  Hessian determinant and its properties regarding the topological degree and the change of variable formula.  In Section \ref{iso}, we define and establish the bounded extrinsic  property for the graphs of $C^{1,\alpha}$ functions whose very weak Hessian determinants are nonnegative.  We will prove Theorem \ref{iso-bou-cur} in Section \ref {iso-case}. Finally, Section \ref{sec-full-plane} is dedicated to a short proof of Theorem \ref{full-plane}. 

\medskip

\subsection {Acknowledgments.}
This project was based upon work supported by the National Science Foundation and was partially funded by the Deutsche Forschungsgemeinschaft (DFG, German Research Foundation). M.R.P. was supported by the NSF award DMS-1813738 and by the DFG  via SFB 1060 - Project Id  211504053.  This project was completed while the author was visiting Institut \'Elie Cartan de Lorraine at Nancy supported by CNRS funding, and he is grateful for their respective hospitality and support.

\section{Properties of the very weak Hessian in higher regularity regimes}\label{MA-sec}

In this section we gather some statements on the properties of $\Det \mathcal{D}^2 v$ in our H\"older regime of regularity and draw some useful conclusions.  The first statement  is in the same line of the known results on the compensated regularity  of the distributional  Jacobians dating back to Wente \cite{W69} and spanning many important contributions, in particular \cite{WY99a, BN11} concerning  fractional regularity. But it is rather formulated for the very weak Hessian determinant operator and the proof uses its structure.  For definitions and statements regarding
Besov spaces  $B^s_{p,q}$ we refer to \cite{RuSi, Tri06}.
\begin{lemma}\label{compens}
Let $V \subset \R^2$  be an open smooth bounded domain, $\alpha \in (1/2, 1)$ and $v\in C^{1,\alpha}(\overline V)$. Then 
\begin{equation}\label{compens-eq}
\Det {\mathcal D}^2  v \in B^{2\alpha-2}_{\infty, \infty}(V). 
\end{equation}
 \end{lemma} 
\begin{proof} 
We can extend $v\in C^{1,\alpha}(\overline V)$ to $\tilde v \in C^{1,\alpha}(\R^n)$ such that 
$$
\|\nabla \tilde  v\|_{0,\alpha} \le C \|v\|_{1,\alpha;V}.
$$ For  a standard 2d mollifier $\psi\in C^\infty_c(B_1(0))$ and the sequence $\psi_\e (x) =
\e^{-2}\psi(x/\e)$, we define  $\tilde v_\e:=\tilde v \ast \psi_\e$ and $(\nabla \tilde v\otimes \nabla  \tilde v)_\e  :=(\nabla \tilde v\otimes \nabla  \tilde v) \ast \psi_\e$ on $V$. We calculate for any $0<\e\le 1$  
$$
\begin{aligned}
\| (\Det {\mathcal D}^2 \tilde v) \ast \psi_\e\|_0 & \le  \|\mbox{curl}\,\mbox{curl} (\nabla \tilde v \otimes \nabla \tilde v)_\e\|_0 
\\ & \le     \|\mbox{curl}\,\mbox{curl} (\nabla \tilde v \otimes \nabla \tilde v)_\e - \mbox{curl}\,\mbox{curl} (\nabla \tilde v_\e \otimes \nabla \tilde v_\e) \|_0  +\| \mbox{curl}\,\mbox{curl} (\nabla \tilde v_\e \otimes \nabla \tilde v_\e)\|_0  
\\& \le C (\e^{2\alpha-2} \|\nabla \tilde v\|^2_{0,\alpha}  + \|\det \nabla^2 \tilde v_\e\|_0 ) \le C \|v\|^2_{1,\alpha;V} \e^{2\alpha-2}, 
\end{aligned} 
$$ where according to \cite[Lemma 1]{CDS} (see also \cite[Lemma 4.3]{LP17}) we used the commutator estimate 
\begin{equation}\label{com-est}
 \| (\nabla \tilde v \otimes \nabla \tilde v)_\e  -    (\nabla \tilde v_\e \otimes \nabla \tilde v_\e)\|_{k}\leq   C  \e^{2\alpha-k}\|\nabla \tilde v\|^2_{{0,\alpha}}  
\end{equation}   and  the estimate  
$$
\|\nabla^2 \tilde v_\e\|_0 \le C\e^{\alpha-1} \|\nabla \tilde v\|_{0,\alpha}.
$$  Hence we obtain  from \cite[Corollary 1.12]{Tri06} that $\Det {\mathcal D}^2 \tilde v \in B^{2\alpha-2}_{\infty, \infty}(\R^n)$. Restricting this distribution to $V$ yields \eqref{compens-eq}  according to \cite[Definition 2.4.1/2]{RuSi}.
\end{proof}

Now, note that when $\alpha>2/3$, $C_c^{0,\alpha} ( V)$ densely embeds into $W_0^{2-2\alpha, 1}(V)$, since $2-2\alpha<\alpha$, and that
by \cite[Proposition 2.1.5]{RuSi}
\begin{equation}\label{dlty}
B^{2\alpha-2}_{\infty, \infty}(\R^n) = (B^{2-2\alpha}_{1,1}(\R^n))' = (W^{2-2\alpha, 1}(\R^n))'.
\end{equation}    As a consequence, by an extension argument,  $B^{2\alpha-2}_{\infty, \infty}(V)  \hookrightarrow (C_c^{0,\alpha} ( V))'$ and the action of $\Det {\mathcal D}^2  v $ any element of $C_c^{0,\alpha} (V)$ is defined. Remember that for a continuous function $u:\overline U\to \R^2$,  $\deg(u, U, y)$  is its Brouwer degree  at a point $y\in \R^2 \setminus  u(\partial U)$. We now  state a slightly more general formulation of the degree formula from \cite{LP17}. 
 
\begin{proposition}\label{deg}
Let $\Omega \subset \R^2$  be an open domain.
Assume that  
$ 2/3 <\alpha <1$,  $v\in C^{1,\alpha}(\Omega)$. 
For  $\delta\in \R$, and $x= (x_1, x_2) \in \Omega$ set 
$
u^\delta(x_1,x_2):= \nabla v(x_1, x_2) + \delta(-x_2, x_1).
$
Let $U\Subset \Omega$ be an open set. Then for every\footnote{$~$An approximation argument through Dynkin's $\pi$-$\lambda$ theorem yields the same result for  $g\in L^\infty(\mathbb{R}^2)$ with $\mathrm{supp }~ g\subset \R^2\setminus \nabla v(\partial U)$. We will not need this fact.}  $g\in C^\infty_c( \R^2\setminus  u^\delta (\partial U))$ the following formula holds true:
 
\bee\label {area}
\int_{\R^2} g(y) \deg(u^\delta, U, y)~ \mathrm{d}y =  \Det {\mathcal D}^2 v [  g \circ  u^\delta]  + \delta^2 \int_U (g \circ  u^\delta) \, \mathrm{d}x. 
\eee
 \end{proposition}

 \begin{proof}
Note that when $u\in C^{0, \alpha}(\Omega, \R^2)$, $\alpha>1/2$, the distributional Jacobian derivative ${\rm J}(u):={\rm Det}(\nabla u)$  is well-defined  as an element of $\mathcal D'(\Omega)$ and  ${\rm J}(u_k) \rightarrow {\rm J} (u)$ if $u_k \to u$ in  $C^{0, \alpha}$-norm \cite[Corollary 1]{BN11}. It can hence be shown  by approximation that  for $v\in C^{1,\alpha}(\Omega)$ in the same regime of regularity we have
$$
{\rm J}(u^\delta) = \Det {\mathcal D}^2 v + \delta^2.
$$ As a consequence,  when $v\in C^{1,\alpha}$ and $\alpha>2/3$,  \eqref{area}  can be also deduced from similar degree formulas for the distributional Jacobian,  see \cite[Lemma 3.1]{LPS21} and also \cite{Olb2, GO20}. We leave the details to the reader. 
 \end{proof}
  
\begin{corollary}\label{convex1}
Let $\Omega \subset \R^2$  be any open domain, $U\Subset \Omega$ be  an open set. Assume that $\alpha, v$ are as in Proposition
\ref{deg} and that 
$$
f:= \Det {\mathcal D}^2  v\ge 0 \quad \mbox{in}\,\, \mathcal{D}'(\Omega).
$$  Then  
\begin{itemize}
\item[(i)] $\deg(\nabla v, U, \cdot) \ge \rchi_{\nabla v (U) \setminus \nabla v (\partial U)}$
\item[(ii)] $\deg(\nabla v, U,\cdot)\in L^1(\mathbb{R}^2 \setminus \nabla v(\partial U))$ and:
$$\int_{\R^2 \setminus \nabla v(\partial U)} \deg(\nabla v, U, y)~
\mathrm{d}y\leq  \mu_f (U). $$
\end{itemize}
\end{corollary}

\begin{proof}
 Let $u^\delta$ be defined as in Proposition \ref{deg}.  We recall that $\deg(u^\delta, U, \cdot)$ is well-defined and
constant on each connected component $\{W_i\}_{i=0}^\infty$ of
$\R^2\setminus u^\delta(\partial U)$, and that it equals $0$ on $\R^2\setminus u^\delta(\overline U)$.  It follows from (\ref{area}) and the nonnegativity of $f$  that for any $\delta >0$ we have 
$$
\deg(u^\delta, U, \cdot) \ge \rchi_{u^\delta (U) \setminus u^\delta(\partial U)}.
$$ Now, let $ \nabla v(x)= y\in \nabla v(U)\setminus \nabla v(\partial U)$ for some $x\in U$. Choose $r>0$ such that $B(y, 2r) \subset \R^2\setminus \nabla v(\partial U)$ and $\delta_0$ small enough for which $\|u^\delta - \nabla v\|_0 <r$ for all $\delta<\delta_0$. For all such $\delta$ we have $B(y,r) \subset \R^2 \setminus u^\delta(\partial U)$, which implies that $\deg(u^\delta, U, \cdot)$ is constant on $B(y,r)$.   But $|u^\delta (x) - \nabla v(x)| < r$, i.e.\@ $u^\delta (x)  \in B(y,r)$, and hence 
$$
\deg(u^\delta, U, y) = \deg(u^\delta, U, u^\delta(x)) \ge 1.  
$$ Therefore, passing to the limit as $\delta\to 0$ and using \cite[Proposition 2.1]{Kav} we conclude with (i).

To show (ii), we apply (\ref{area}) for $\delta =0$ to an increasing sequence of   test 
functions $g_k \in C^\infty_c (\R^2\setminus\nabla v(\partial U))$ that converge pointwise to
$\rchi_{\R^2\setminus\nabla v(\partial U)}$. In view of the monotone convergence
theorem both sides of (\ref{area}) converge by (i),  from which we obtain in the limit as $k\to\infty$:

\bee\label {area3}
\int_{\R^2\setminus\nabla v(\partial U)} \deg(\nabla v, U, y)~
\mathrm{d}y = \mu_f\Big (U\setminus(\nabla v)^{-1}(\nabla v(\partial U))\Big)  \leq \mu_f(U).
\eee  \end{proof}   
\begin{proposition}\label{deg2}
Let $\Omega \subset \R^2$  be an open domain, $\alpha \in (1/2, 1)$, $v\in C^{1,\alpha}(\Omega)$. Assume that    
$$
\Det{\mathcal D}^ 2 v =f  \ge 0 \quad \mbox{in}\,\, \mathcal{D}'(\Omega),
$$    and that $U\Subset \Omega$ is  piecewise smooth.  Then
 \bee\label{area2}
\int_{\R^2} \deg(\nabla v, U, y)~ \mathrm{d}y = \mu_f(U).
\eee 
 \end{proposition}
 \begin{remark}
The statement is independent of Corollary \ref{convex1}-(ii). Here, the seemingly stronger \eqref{area2} is merely stated for a piecewise smooth domain $U$, while it is valid for a better range $\alpha>1/2$.  The inequality in Corollary \ref{convex1}-(ii)   is established for all open sets $U$, but only for $\alpha>2/3$, as its proof  uses the localized version of the degree formula \eqref{area}. Indeed it is an open problem whether \eqref{area} and Corollary \ref{convex1}-(ii)  hold for $\alpha>1/2$,  even when  $\delta=0$ and  $\Det{\mathcal D}^ 2 v$ is regular enough for the trilinear expression $\Det {\mathcal D}^2 v [  g \circ  \nabla v]$  to be well-defined.
 
Also note that when $\partial U$ is piecewise smooth, and $\alpha>1/2$, $\nabla v(\partial U)$ is of measure zero \cite[Lemma 4]{CDS}, and $deg(\nabla v, U, y) \in L^1(\R^2)$  \cite{Olb2}. This is no more valid when $U$ is arbitrary, and hence we shall remove 
$\nabla v(\partial U)$ from the domain of integration in Corollary \ref{convex1}-(ii). 
 \end{remark}
 
\begin{proof}
The proposition is a consequence of the similar degree formula for the distributional Jacobian ${\rm J}(\nabla v) = \Det \mathcal{D}^2 v$, when $v\in C^{1,\alpha}$ and $\alpha>1/2$, see \cite[Theorem 3 and its proof]{GO20}.
\end{proof}

\section{The graph $S=\xi_v(\Omega)$ as a surface of bounded extrinsic curvature.}\label{iso}

We denote by $S= \xi_v(\Omega)$ the graph of $v\in C^{1,\alpha}(\Omega)$, that is the $ C^{1,\alpha}$ surface that is the  image of 
$\xi_v(x) = (x, v(x))$.  Let  $N=\vec n\circ \xi_v\in C^{0, \alpha}(\Omega, \mathbb{S}^2)$, and we note that 
up to the choice of the orientation of $\vec n$,  and for $\eta:\mathbb{R}^2\to \mathbb{S}^2$, defined by :
\begin{equation}\label{eta}
\eta(x) = \frac{1}{\sqrt{1+ |x|^2}} (x, -1),
\end{equation} we have $N=\eta\circ\nabla v$.

\begin{proposition}\label{convex2}
Let $\Omega \subset \R^2$ be an arbitrary open set, and $\alpha, v$ be as in Proposition \ref{deg}.  Assume that 
$$
f:= \Det {\mathcal D}^2  v\ge 0.
$$  Then  for every finite collection $\{E_i\}_{i=1}^k$ of closed subsets of
$\Omega$ which are pairwise disjoint:
\begin{equation}\label{gradvBV}
\sum_{i=1}^k |\nabla v(E_i)|  \leq  \mu_f(\Omega).
\end{equation}  Moreover,   if  $\mu_f(\Omega) < \infty$,  the surface $S= \xi_v(\Omega)$ is a surface of bounded extrinsic
curvature, and is of nonnegative curvature.  Moreover, if $f \not \equiv 0$, then the positive curvature of $S$ is non-zero. 
\end{proposition}

 \begin{proof}
 
Let $\{E_i\}_{i=1}^k$  be a finite collection of pairwise disjoint closed subsets of $\Omega$.   Approximating each $E_i$ from within  with an increasing sequence of compact sets,  we note that the value of $ |\nabla v(E_i)|$
 is the limit of  evaluations on the given sequence of compact subsets. Therefore, it is sufficient to assume that each $E_i$ is compact in order to establish  \eqref{gradvBV} by a passing to the limit argument for the more general case. 
 
 \medskip
 
 If each $E_i$ is compact, there are pairwise disjoint open smooth sets $U_i\supset E_i$, compactly contained in $\Omega$, for $i=1\ldots k$.  
Noting the fact that since $U_i$ are smooth and $\alpha>1/2$, we have $|\bigcup_{i=1}^k\nabla v(\partial
U_i)|=0$ \cite[Lemma 4]{CDS}. Corollary \ref{convex1} yields
\begin{equation*}
\begin{split}
\sum_{i=1}^k |\nabla v(E_i)| & \leq \sum_{i=1}^k \int_{\R^2}
\rchi_{\nabla v(U_i)} = \sum_{i=1}^k \int_{\R^2\setminus \nabla
  v(\partial U_i)} \rchi_{\nabla v(U_i)\setminus \nabla
  v(\partial U_i)} \\ & \leq \sum_{i=1}^k \int_{\R^2\setminus \nabla
  v(\partial U_i)}\deg(\nabla v, U_i,y)~\mbox{d}y \leq 
 \sum_{i=1}^k \mu_f(U_i)  \leq \mu_f(\Omega),
\end{split}
\end{equation*}
where we used the nonnegativity of $\mu_f$ in the last inequality.  
Hence (\ref{gradvBV}) is established.

\medskip

Since the map $\eta$ in (\ref{eta}) is smooth, and $|\partial_1 \eta \times \partial_2 \eta| \le 1$,
(\ref{gradvBV}) implies $\sum_{i=1}^k {\mathcal H}^2(N(E_i)) \leq \mu_f(\Omega)$  for any collection $\{E_i\}_{i=1}^k$ as specified. As a straightforward conclusion, $S$ has the bounded extrinsic curvature property  if $\mu_f(\Omega)<\infty$ and the absolute curvature of  $S$ is bounded by $\mu_f(\Omega)$.
\medskip
 
By \cite[Lemma, p.\@ 594]{Po73}, the topological index of $\vec n$ at an elliptic, parabolic or hyperbolic
point $p=\xi_v(x)\in S$, which we denote by $i(p)$, equals, respectively, $+1, 0, -1$, and the index of a flat point is less than $-1$.  Consider the contour $\gamma$ lying on the surface $S$ and encircling the point $p$ as described on \cite[p.\@ 595]{Po73}, and its projection on $\Omega$, encircling the point $x$, which is parameterized as the simple closed curve $\Gamma : \mathbb {S}^1\to \R^2$. Since $p$ is a regular point, we can fix $r>0$ and choose $\gamma$ such that the image of $\Gamma$ lies in $B_r(x)$, when $O:=\xi_v(B_r(x))$ satisfies (\ref{reg_def}). A careful examination of the geometric definition of the index on  \cite[p.\@ 595]{Po73} in our setting, considering the relationship  $\vec n \circ \xi_v = \eta \circ \nabla v$, shows that the index of $\vec n$ at a point $p$ equals the total change in the angle induced by the mapping 
$$
z: \mathbb{S}^1 \to  \mathbb{S}^1 \quad a \to z(a):=\frac{\nabla v(\Gamma(a)) - \nabla v(x)}{|\nabla v(\Gamma(a)) - \nabla v(x)|}
$$ divided by $2\pi$, when $a$ makes one counter-clockwise turn over the circle $\mathbb {S}^1$.  In other words, 
$$
i(p) = \frac{1}{2\pi}(\theta(1) - \theta(0)), 
$$ when $\theta :  [0,1] \to \R$ is the lifting of the map $ \tilde z(t):=z(e^{2\pi it})$ to $\R$, satisfying $\tilde z(t)= e^{i \theta(t)}$.  It is a well known fact that the latter quantity is  the topological degree of the mapping $z$ as a continuous mapping of the unit circle into itself (see e.g.\@ \cite[Exercice 6, p.\@ 113]{OR09}), therefore $i(p)= \deg(z)$.
\medskip 

Now we approximate $\Gamma$ uniformly with a sequence of  regular simple closed curves $\Gamma_k :\mathbb{S}^2 \to \R^2$ with images inside $B_r(x)$, and define 
$$ 
z_k: \mathbb{S}^1 \to  \mathbb{S}^1 \quad a \to  z_k(a):= \frac{\nabla v(\Gamma_k(a)) - \nabla v(x)}{|\nabla v(\Gamma_k(a))- \nabla v(x)|}.
$$ As $z_k$ converges uniformly to $z$, then ${\rm deg}(z_k)$ converges to ${\rm deg}(z)$, and degree being an integer value, we obtain that  $i(p)= {\rm deg}(z_k)$  for $k$ large enough. Fix such $k$, and let $X= \Gamma_k(\mathbb{S}^1)$.  Then $X \subset B_r(x)$ is a 1-dimensional smooth simple curve, with $X= \partial D$, and $x\in D\subset B_r(x)$. Since $\Gamma_k$ is a diffeomorphism between  $\mathbb S^1$ and $X$,  ${\rm deg}(z_k)$ equals to the degree of the mapping
$$
w: X \to \mathbb S^1 \quad y \to w(y):=  \frac{\nabla v(y) - \nabla v(x)}{|\nabla v(y) - \nabla v(x)|}
$$ defined on $X$, i.e.\@ $i(p)= \deg (w)$.  However, by \cite[Porposition IV.4.5]{OR09}, we have
$$
\deg(w) = \deg(\nabla v, D, \nabla v(x)).
$$  Since $\xi_v(B_r(x))$ satisfies \eqref{reg_def},  for all $y \in B_r(x)$, $\nabla v(y) \neq \nabla v(x)$ for $y\neq x$, which finally yields
$$
i(p) = \deg(\nabla v, D, \nabla v(x)) = \deg(\nabla v, B_{r/2}(x), \nabla v(x)), 
$$ since $\nabla v (x) \notin \nabla v(\overline D\setminus B_{r/2}(x))$   \cite[Corollaire 2.4]{Kav}.
It follows from Corollary \ref{convex1}(i) that  any regular point $p\in S$  is elliptic. 
This implies that $S$ is a surface of nonnegative curvature. 

\medskip

For $z\in \mathbb{S}^2$, and $A\subset S$, let $m_A(z)$  be  the cardinality of the set $A\cap {\vec n}^{-1}(z)$.  By \cite[p.\@ 577, Lemma 3]{Po73}, the function $m_A$ is 
measurable. It follows then from \cite[p.\@ 590, Theorem 3]{Po73} that the absolute curvature of $A$ equals $\int_{\mathbb{S}^2} m_A(z)~\mbox{d}s(z)$, for every Borel subset $A$ of $S$.  On the other hand, for any  open set $O\subset S $ and any finite collection of pairwise disjoint closed sets $\{F_i\}_{i=1}^k$ in $O$, we have for $E_i = \xi_v^{-1}(F_i)$:
$$
\sum_{i=1}^k \mathcal{H}^2 (\vec n (F_i)) =\sum_{i=1}^k \mathcal{H}^2 (N(E_i)) \le \sum_{i=1}^k  |\nabla v(E_i)|  \le    \mu_f(\xi_v^{-1}(O)). 
$$ Hence, by the definition of the absolute curvature, for any Borel set $A\subset S$ we obtain
$$
\ds \int_{\mathbb{S}^2} m_A(z)~\mbox{d}s(z) \le    \mu_f(\xi_v^{-1}(A)), 
$$ which implies for $A=S$ that 
\begin{equation}\label{mult-ifty=0}
\mu_f(\Omega)<\infty \Longrightarrow {\mathcal H}^2(m_S^{-1}(+\infty)) = 0.
\end{equation}  

If $f$ is nonzero, then for some open disk $U\subset \Omega$, $\mu_f(U)>0$. We claim that $m_S(N(x))< \infty$ for some $x\in U$. 
If, by contradiction,  $m_S(N(x))=+\infty$ for every $x\in U$,  then $N(U) \subset m_S^{-1}(+\infty)$. 
It follows from  finiteness of $\mu_f$ and \eqref{mult-ifty=0} that ${\mathcal H}^2 (N(U))\le {\mathcal H}^2(m_S^{-1}(+\infty))  =0$, yielding that $|\nabla v(U)|=0$, since the mapping $\eta$ in (\ref{eta}) is a smooth diffeomorphism between $\R^2$ and $\eta(\R^2) \subset \mathbb {S}^2$. Thus, once again using the fact that $|\nabla v (\partial U)|=0$ \cite[Lemma 4]{CDS},
(\ref{area2})  yields:
$$ \mu_f(U)= \int_{\R^2\setminus \nabla v(\overline U)}\deg(\nabla v, U, y)~\mbox{d}y=0, $$
contradicting the assumption on $U$.  Hence, there exists $x\in U$ such that $m_S(N(x))<+\infty$, yielding that 
$p= \xi_v(x)$ satisfies \eqref{reg_def}, and so it is a regular point. We already showed that any regular point must be elliptic.   
Since $S$ contains elliptic points, \cite[Theorem 12, p. 600]{Po73} implies that the positive curvature of $S$ is non-zero. \end{proof}
   
       \section{Isometric immersions of surfaces of nonnegative curvature}\label{iso-case}
      
        \subsection{Some properties of the distributional curvature}\label{iso-case-prel}
     
     Throughout this section, $\psi\in C^\infty_c(B_1(0))$ is a  standard 2d mollifier, with $\psi \ge 0$ and $\int_{\R^2} \psi =1$. For the sequence $\psi_\e (x) = \e^{-2}\psi(x/\e)$, and any function or mapping defined on an open set $\Omega\subset \R^2$, we denote, with an abuse of notation vis-\`a-vis $\psi$ itself, the mollification $f\ast \psi_\e$ by using the subscripted $f_\e$.  We will need the following consequence of \cite[Lemma 1]{CDS}:
 
 \begin{lemma}\label{3-commute}
 Let $\Omega \subset \R^2$ be an open set, $V\Subset \Omega$, and $f,g,h \in C^{0,\alpha}(\Omega)$.    Then
 $$
 \|(fgh)_\e - f_\e g_\e h_\e\|_{1;V} \le C (\|f\|_{\alpha;\Omega},  \|g\|_{\alpha;\Omega}, \|h\|_{\alpha;\Omega}) \e^{2\alpha-1}.
 $$
 \end{lemma} 
 
 \begin{proof} The estimate is obtained by an iterated use of  
\begin{equation}\label{2-commute}
 \|\nabla a_\e\|_{0;V} \le C \e^{\alpha-1},\quad \|(ab)_\e - a_\e b_\e \|_{j;V} \le C \e^{2\alpha-j}.
\end{equation}  for $a,b \in C^{0,\alpha}(\Omega)$ and $j=0,1$, as proved in \cite[Lemma 1]{CDS}.  The details are left to the reader.
 \end{proof}

  \begin{lemma}\label{g-C1-conv}
Let  $U \subset \R^2$ be an open set and $g_1, g_2 \in  C^1 (U, \R^{2\times 2}_{sym, pos})$. Assume that $V\Subset U$ is an open bounded set and that $\det  g_i \ge \lambda >0$ on $V$. Then for a constant $C:=  C (\|g_1\|_{0;V},  \|g_2\|_{0;V}, \lambda) $ we have
$$
\forall \varphi \in W^{1,1}_0(V) \quad |\kappa_{g_1}[\varphi] - \kappa_{g_2}[\varphi]|  \le C \|\varphi\|_{W^{1,1}(V)} \sum_{j=0}^1 (\|g_1\|_{1;V} + \|g_2\|_{1;V} )^{2-j}   \|g_1 - g_2\|_{j;V}.
$$   
\end{lemma} 
 
The proof is obtained by straightforward  calculations.  We leave the details to the reader. 
 
\medskip
  
We now observe that  the independence of the Gaussian curvature from the coordinate system is still valid in our weaker setting.

\begin{lemma}\label{coord-ch} Assume $U$ and $g$ as in Lemma \ref{g-C1-conv}. Let $\xi:  U' \to   U$ be a $C^2$ smooth diffeomorphism. Let  $\xi^\ast g= (\nabla \xi)^T (g\circ \xi)  \nabla \xi$,  be the  pull-back of the metric $g$  by $\xi$ on $U'$. Then
\begin{equation}\label{chang-var}
\kappa_{\xi^\ast g}= \kappa_g \circ \xi,  
\end{equation} where $\kappa_g \circ \xi$ is defined  through 
$$
\forall \varphi \in C^\infty_c(U') \quad (\kappa_g \circ \xi) [\varphi] := \kappa_g [|J(\xi)|^{-1}\varphi \circ \xi^{-1}].
$$
\end{lemma}
 \begin{proof}
Since $\xi\in C^2$, we have $\psi= |J(\xi)|^{-1}\varphi \circ \xi^{-1} \in  W^{1,1}_0(V)$ with a suitable open set $V$ with ${\rm supp} \, \psi \subset V \Subset U$, and hence, as $g\in C^1(\overline V)$,
$\kappa_g [\frac{1}{|J(\xi)|}\varphi \circ \xi^{-1}]$ is well-defined, and the value is independent of the choice of $V$.
 
\medskip
 
Note the well known fact that the identity is valid for smooth $g\in C^\infty (\overline V)$. Applying Lemma  \ref{g-C1-conv} to a regularizing sequence $g_k\in C^\infty  (\overline V)$ converging to  $g$ in $C^1(\overline V)$, the conclusion is obtained by a straightforward  passage to the limit. 
 \end{proof}
     
The following proposition is an immediate corollary. 

\begin{proposition}\label{chart-cur}
 Let $\Sigma$ be a two dimensional $C^2$ manifold and let $g$ be a $C^1$ Riemannian metric over $\Sigma$. The distributional 
 Gaussian curvature  $\kappa_g$ can be defined in each chart of $\Sigma$ and is invariant by the admissible changes of coordinates through
 the formula \eqref{chang-var},  
 and hence it is a well-defined distribution on $\Sigma$.
 \end{proposition}
   
 In what follows we discuss further properties of $\kappa_g$ when $g$ further enjoys some H\"older regularity. 
 
 \begin{proposition}\label{regularity}
 Let $U\subset \R^2$ be a bounded smooth domain. If $g\in C^{1,\alpha}(\overline U)$ for $1/2<\alpha' <\alpha<1$, then $\kappa_g$ can be uniquely extended as a bounded linear operator over $C_0^{0,\alpha'}(U)$,  and 
 $$
 \forall \varphi \in C_0^{0,\alpha'}(U)\quad   |\kappa_g[\varphi] | \le C(g) \|\varphi\|_{\alpha';U}.
 $$
 \end{proposition}
 \begin{proof}
 First  note that since  $g\in C^{1,\alpha}(\overline U) = B^{1+\alpha}_{\infty, \infty}(U)$ \cite[Proposition 2.1.2]{RuSi}, we have $${\rm curl}\, {\rm curl}\,\, g  \in B^{\alpha-1}_{\infty, \infty}(U)$$ by \cite[Proposition 2.1.4/2]{RuSi}.  According to \cite[Theorem 4.6.1/2(i)]{RuSi}  the product 
\begin{equation}\label{prod-cont0}
B^{\alpha-1}_{\infty, \infty}(U) \cdot B^\alpha_{\infty, \infty}(U)     \hookrightarrow B^{\alpha-1}_{\infty, \infty}(U),  
\end{equation}  
is continuous, which yields $\kappa_g \in B^{\alpha-1}_{\infty, \infty}(U)$, since $(\det g)^{-1} \in C^{0, \alpha}(\overline U)$, 
and $L(g)$ is regular enough.  

\medskip

Now, since $\alpha >1/2$, we can fix $\sigma>0$ such that $1-\alpha<1/2< \sigma <\alpha'$. By \cite[Theorems 2.4.4/1 and  2.4.4/4]{RuSi}, we have the embedding
 $$
 B^{\alpha-1}_{\infty, \infty}(U) \hookrightarrow B^{-\sigma}_{2,2}(U) =  (\mathring{B}^\sigma_{2,2}(U))'  
 $$ On the other hand, also by \cite[Theorem 2.4.4/1]{RuSi}, $C_0^{0,\alpha'}(U) \subset B^{\alpha'}_{\infty, \infty}$ 
 embeds densely in $\mathring{B}^\sigma_{2,2}(U)$, which yields the embedding
 $$
 (\mathring{B}^\sigma_{2,2}(U))'  \hookrightarrow (C_0^{0,\alpha'}(U))'.
 $$  Overall, we obtained
 $$
 \kappa_g \in  B^{\alpha-1}_{\infty, \infty}(U) \hookrightarrow (C_0^{0,\alpha'}(U))'.
 $$ as required. 
 \end{proof}

\begin{corollary}\label{compo-xi} Assume $U$, $\alpha$, $\alpha'$ and $g$ as in Proposition \ref{regularity}. Let $\xi:  U' \to  U$ be a $C^{1,\alpha}$ smooth diffeomorphism. Then the composition $\kappa_g \circ \xi$  is a well-defined distribution and for any sequence of diffeomorphisms $\xi_k \to \xi$  converging in $C^{1,\alpha'}(U')$ to  $\xi$ we have:
$$
\forall \varphi \in C^\infty_c(U')  \quad \lim_{k\to \infty} (\kappa_g \circ \xi_k) [\varphi] = (\kappa_g \circ \xi) [\varphi]. 
$$

\end{corollary}

 \begin{proof}
 It is sufficient to observe that, since the diffeomorphism $\xi\in C^{1,\alpha}(\overline U')$, we have 
 $$
 |J(\xi)|^{-1}\varphi \circ \xi^{-1} \in C^{0,\alpha'}_c(U')
 $$ and hence  the action of  $\kappa_g$ on  $\frac{1}{|J(\xi)|}\varphi \circ \xi^{-1}$ is well-defined. The continuity follows from Proposition \ref{regularity} by observing that 
 $$
 |J(\xi_k)|^{-1}\varphi \circ \xi_k^{-1} \xrightarrow{k\to \infty}|J(\xi)|^{-1}\varphi \circ \xi^{-1} \quad \mbox{in}\,\,  C^{0,\alpha'}(U'). 
 $$ \end{proof}
  
 In view of the above, we would have liked to claim that $\kappa_{\xi^\ast g} = \kappa_{g} \circ \xi$ when $\xi$ is merely of $C^{1,\alpha}$ regularity, but our observations do not guarantee that the  distributional Gaussian curvature of $\xi^\ast g$ is well-defined, as it might lack the assumed  $C^1$ regularity. We will analyze hence the regularization of the pull-back metrics.  

\begin{lemma}\label{asym-pull}
Assume $U \subset \R^2$ is an open set and $\ds \alpha>1/2$, and $g \in C^{1,\alpha} (\overline U, \R^{2\times 2}_{sym, pos})$.  Let $\xi : U' \to U$ be a $C^{1,\alpha}$  diffeomorphism and 
$$
\xi^\ast g:= (\nabla \xi)^T (g\circ \xi)  \nabla \xi
$$ be the pull-back metric. Let $(\xi^\ast g)_\e$ and $\xi_\e$ be the respective regularizations of $\xi^\ast g$ and $\xi$. If $x'\in U'$, there exists a disk $B' \Subset U'$ centered at $x'$ such that for all $\e$ small enough, $\xi_\e$ is a diffeomorphism on $\overline B'$,   
$$\ds \inf_\e \inf_{B'} \det (\xi_\e^\ast g) >0,
$$ 
$$
\|(\xi^\ast g)_\e \|_{j;B'} + \|\xi_\e^\ast g \|_{j;B'}  \le C(\|\xi\|_{1,\alpha}, \|g\|_{1,\alpha})    \e^{j(\alpha-1)}, \quad j\in \{0,1\}.
$$
and 
$$
\|(\xi^\ast g)_\e - \xi_\e^\ast g    \|_{j;B'} \le C(\|\xi\|_{1,\alpha}, \|g\|_{1,\alpha})  \e^{(1+j)\alpha-j},\quad j\in \{0,1\}.
$$
\end{lemma}
\begin{proof}
Noting the uniform convergence  of $\nabla \xi_\e$ to $\nabla \xi$ on a neighborhood of $x'$,    the inverse function theorem can be applied uniformly on a ball $B'$ centered at $x'$, to obtain that each $\xi_\e$ is a diffeomorphism on $B'$ and $\det \xi_\e^\ast g$ is uniformly bounded from below on $B'$.  It is straightfoward that $(\xi^\ast g)_\e$ and $\xi_\e^\ast g$ are uniformly $C^0$-bounded, and the estimates on their $C^1$ norms  follow immediately from the $C^{1,\alpha}$ regularity of 
$\xi$ and $g$, via the first estimate in \eqref{2-commute}.

\medskip 
We now estimate, letting $F:=\nabla \xi \in C^{0,\alpha}(\Omega)$:
$$
\begin{aligned}
\|(\xi^\ast g)_\e - \xi_\e^\ast g \| _{j;B'}  & \le  \|\underbrace{(\xi^\ast g)_\e - F_\e^T (g\circ \xi)_\e F_\e}_{I_1}\| _{j;B'} 
+ \|\underbrace{F_\e^T (g\circ \xi)_\e F_\e  - F_\e^T (g \circ \xi_\e) F_\e}_{I_2}\| _{j;B'}   
\end{aligned}
$$
The first term is estimated using Lemma \ref{3-commute}:
 $$
\|I_1 \|_{j;B'}=  \Big \| \Big (F^T  (g\circ \xi) F\Big )_\e - F_\e^T (g\circ \xi)_\e F_\e\Big  \| _{j;B'} \le  C\e^{2\alpha-j} \le C \e^{(1+j)\alpha-j}.
 $$ For the second term, we first establish an  estimate on the $C^0$ norm: 
 $$
 \begin{aligned}
  \|I_2\| _{0;B'} & \le  C
  \| (g\circ \xi)_\e    -   g \circ \xi_\e  \| _{0;B'} 
   \\  & \le   \| (g\circ \xi)_\e    -   g \circ \xi  \| _{0;B'}  +   \| g\circ \xi_\e    -   g \circ \xi  \| _{0;B'} 
  \\ &  \le C \e \le C\e^\alpha, 
 \end{aligned} 
 $$since both $\xi$ and $g$ are $C^1$. 
For the $C^0$ norm of the derivatives of $I_2$ we have the upper bound
 $$ 
 \begin{aligned}
   \|\nabla  I_2    \| _{0;B'} & \le C \|  \nabla F_\e \|_{0;B'}  
 \| (g\circ \xi)_\e    -   (g \circ \xi_\e)  \| _{0;B'}    + C  \Big \| \nabla  \Big (  (g\circ \xi)_\e    -   g \circ \xi_\e \Big ) \Big \| _{0;B'}
 \\ & \le C \e^{\alpha-1 }\e + \Big \|\Big  (\nabla  (g\circ \xi) \Big)_ \e - \nabla (g\circ \xi) \Big \|_{0;B'} + \| \nabla (g\circ \xi)  - \nabla (g\circ \xi_\e) \|_{0;B'}
  \\ & \le C \e^{2\alpha-1}  + C\e^\alpha+ \|(\nabla g \circ \xi)  \nabla \xi - (\nabla g \circ \xi_\e) \nabla \xi_\e\|_{0;B'},
 \end{aligned} 
 $$ where we used the fact that $g\circ \xi \in C^{1,\alpha} (\overline U)$.  It remains to estimate the last term in order to conclude. We have
  $$
  \begin{aligned}
\|(\nabla g \circ \xi)  \nabla \xi - (\nabla g \circ \xi_\e) \nabla \xi_\e\|_{0;B'} 
& \le  \| ( \nabla g \circ \xi  - \nabla g \circ \xi_\e  ) \nabla \xi \|_{0;B'} 
\\ &  + \|(\nabla g \circ \xi_\e) ( \nabla \xi -   \nabla \xi_\e)\|_{0;B'}
\\ & \le  C \|(\nabla g \circ \xi)   - (\nabla g \circ \xi_\e)   \|_{0;B'}  + C  \| \nabla \xi - \nabla \xi_\e\|_{0;B'}
\\ & \le  C \|\xi   - \xi_\e   \|^{\alpha}_{0;B'}  + C\e \le C \e^{\alpha} \le C\e^{2\alpha-1}, 
 \end{aligned} 
 $$ where once again we used the $\alpha$-H\"older continuity of $\nabla g$. \end{proof}
 
  \subsection{Proof of Theorem \ref{iso-bou-cur}.} 
  Let $u: (\Omega, g) \to \R^3$ be given according to the assumptions of Theorem \ref{iso-bou-cur}, and $S'$ be given as a surface compactly contained in $S$. Consider a open set $\Omega''$, such that $S' \Subset u(\Omega'') \Subset S$, and let $S'':= u(\Omega'')$. For the first assertion, it is enough to show that for any $p\in S''$, the area induced by the Gauss map $\vec n$  on an open neighborhood of $p$ in $S''$ has finite total variation. A covering argument will then imply that $S'$ is a surface of bounded extrinsic curvature.  Without loss of generality, we can assume that $p=0$, and that $\vec n (p)= \vec e_3$, so that the tangent plane $T_pS$ is identified with $\R^2 \times \{0\}$.  
 
 \medskip
 
  Let $O$ be an open neighborhood of $p=0$ in $S''$, and let $\Omega_0:=u^{-1}(O)$, $x_0= u^{-1}(p)$. 
 We write $u:= (\tilde u, u_3)$, where $\tilde u := (u_1, u_2) : \Omega_0 \to \R^2$.  We note that 
$$
\frac{\partial_1 u \times \partial_2 u}{|\partial_1 u \times \partial_2 u|} (x_0)  = \vec n (x_0) = \vec e_3, 
$$ which yields that at $x_0$
$$
\det \nabla \tilde u  = (\partial_1 \tilde u)^\perp \cdot \partial_2  \tilde u =   (\vec e_3 \times \partial_1 u)\cdot \partial_2 u = (\partial_1 u \times \partial_2 u)\cdot \vec e_3 > 0
$$ since $u$ is an immersion. Implicit function theorem implies that there exists open sets $U\subset \Omega_0$, and $U' = \tilde u(U) \ni 0$, such that $\tilde u: U \to U'$ is a $C^1$ diffeomorphism.  Let $\xi \in C^1(U', U)$ be the inverse of $\tilde u$. If necessary by shrinking $U$ and $U'$, we can assume that $\tilde u \in C^{1,\alpha} (\overline U, \overline U')$ and $\det \nabla \tilde u >b_0>0$ in $\overline U \subset \Omega$.
\begin{lemma}
$\xi$  belongs to $C^{1,\alpha}(\overline{U'}, \R^2)$. 
\end{lemma}  
\begin{proof}
We estimate for any $x,y\in U$, $F= \nabla \tilde u \in C^{0,
\alpha}(\overline U, \R^{2\times 2})$ and $b:= \det \nabla \tilde u \in  C^{0, \alpha}(\overline U) $
 $$
 \begin{aligned}
\Big |(\nabla \tilde u)^{-1} (x)-(\nabla \tilde u)^{-1} (y)  \Big |   &  =  \Big |b^{-1} {\rm Cof} (F) (x)-  b^{-1} {\rm Cof} (F)(y)\Big | 
 \\ & \le  \frac {\|F\|_{0;U}} {b_0^2}| b (x) - b (y)| +  \frac {\|b\|_{0;U}} {b_0^2}| F (x)-  F (y)|
 \\ &  \le C(\|F\|_{0,\alpha;U}, \|b\|_{0,\alpha;U}, b_0)  |x-y|^\alpha
 \end{aligned}
 $$ Hence for any $x',y'\in U'$
 $$
\begin{aligned}
|\nabla \xi (x')-\nabla \xi(y')|  
& \le  \Big |(\nabla \tilde u)^{-1} (\xi (x'))-(\nabla \tilde u)^{-1} (\xi(y')) \Big |
\\ & \le   C(\|\tilde u\|_{1,\alpha;U}, b_0)  |\xi(x') - \xi(y')|^\alpha
\\ & \le C(\|\xi\|_{1;U}, \|\tilde u\|_{1,\alpha; U}, b_0)  |x'-y'|^\alpha.
 \end{aligned}
 $$
\end{proof}

To proceed we define
$$
v: U' \to \R, \quad v:= u_3 \circ \xi. 
$$ Note that $v\in C^{1,\alpha}(\overline{U'})$, $v(0) = 0$, $\nabla v(0) = 0 \in \R^2$. Consider the graph mapping 
$$
\xi_v : U' \to O, \quad \xi_v  (x') := (x', v(x')), 
$$  and note that 
$$
\xi_v = u \circ \xi,
$$ and that $\xi_v(U')=u(U) \subset O \subset S$ is the graph of   $v$ over $U'$. We obtain hence the identity
$$
g':= {\rm Id} + \nabla v \otimes \nabla v  =  (\nabla \xi_v)^T \nabla \xi_v = \nabla (u \circ \xi)^T  \nabla (u\circ \xi) =   \nabla \xi^T (g \circ \xi) \nabla \xi = \xi^\ast g, 
$$  where we used the fact that $u: (\Omega, g) \to \R^3$ is an isometry, i.e.\@
$$
\nabla u^T \nabla u =g.
$$

From now on we use the fact that by assumption $\alpha> 2/3$. Fix $ 2/3 <  \beta < \alpha$.
By Lemma \ref{compens}  we have 
$$
\Det {\mathcal D}^2  v \in B^{2\beta-2}_{\infty, \infty}(U').
$$ Also note that $(1+ |\nabla v|^2)^2 \in C^{0,\beta}(\overline U') = B^\beta_{\infty, \infty}(U')$.    
Since $\beta+ (2\beta-2) >0$, we deduce that the distributional product 
$$
\tilde \kappa_{g'}:= (1 + |\nabla v|^2)^{-2} (\Det {\mathcal D}^2 v)
$$ is well defined as an element of $B^{2\beta-2}_{\infty, \infty}(\overline U')$ \cite[Theorem 4.6.1/2(i)]{RuSi}.  

\begin{proposition}\label{V'}
$\tilde \kappa_{g'}$  can be extended as a bounded linear operator on $C_c^{0,\alpha}(U')$ and 
\begin{equation}\label{equiv-dist}
 \forall \varphi \in C^\infty_c (U')   \quad \tilde \kappa_{g'}  [ \varphi] =  \Det {\mathcal D}^2 v  [(1 + |\nabla v|^2)^{-2} \varphi]. 
\end{equation}  Moreover,  for any open set $V' \Subset U'$, let $v_\e$ is the usual mollification of $v$, defined on $V'$ for $\e$ small enough.  Then 
\begin{equation}\label{v-mol-conv}
 \forall \varphi \in C^\infty_c (V')  \quad \int_{V'}(1 + |\nabla v_\e|^2)^{-2} (\Det {\mathcal D}^2 v_\e) \, \varphi  \xrightarrow{\e\to 0}  
\tilde \kappa_{g'} [\varphi] . 
\end{equation}
\end{proposition}
 
 \begin{proof}
 
 First note that $C_c^{0,\alpha} (\overline U')$ densely embeds into $W_0^{2-2\beta, 1}(U')$, since $2-2\beta<\alpha$, which together with  the duality \eqref{dlty}, imply the first assertion through applying an extension argument.   
\medskip

To proceed, and before  proving \eqref{equiv-dist}, we first show  the last assertion \eqref{v-mol-conv}.  First we claim that
 $$
 \Det {\mathcal D}^2 v_\e \xrightarrow{\e \to 0}  \Det {\mathcal D}^2 v\quad  \mbox{in} \,\, B^{2\beta-2}_{\infty, \infty} (V').
 $$  Indeed we estimate, using \eqref{com-est} and interpolation
 $$
 \begin{aligned}
 \|\Det {\mathcal D}^2  v_\e   - ( & \Det {\mathcal D}^2  v) \ast \psi_\e\|_{B^{2\beta-2}_{\infty, \infty} (V')}
 \le \|{\rm curl} (\nabla  v_\e \otimes \nabla  v_\e) - {\rm curl} (\nabla  v \otimes \nabla  v)_\e\|_{0 ,2\beta-1;V'}
 \\ & \le \|\nabla  v_\e \otimes \nabla  v_\e - (\nabla  v \otimes \nabla  v)_\e\|^{2\beta-1}_{2;V'}  
 \|\nabla  v_\e \otimes \nabla  v_\e - (\nabla  v \otimes \nabla  v)_\e\|^{1-(2\beta-1)}_{1;V'} 
 \\ & \le C\e^{(2\alpha-2)(2\beta-1)+ (2\alpha-1)(2-2\beta)} = C\e^{2(\alpha-\beta)} \xrightarrow{\e\to 0}  0,
\end{aligned}   
 $$  which yields our claim, since as $\Det {\mathcal D}^2 v \in B^{2\alpha-2}_{\infty, \infty} (U')$,  $(\Det {\mathcal D}^2 v) \ast \psi_\e$ converges to $\Det {\mathcal D}^2  v$  in $B^{2\beta-2}_{\infty, \infty} (V')$. Further
$$
 (1 + |\nabla v_\e|^2)^{-2}   \xrightarrow{\e \to 0}  (1 + |\nabla v|^2)^{-2} \quad  \mbox{in} \,\, B^{\beta}_{\infty, \infty} (V').
$$  Now the continuity of the product 
\begin{equation}\label{prod-cont}
B^{2\beta-2}_{\infty, \infty}(V') \cdot B^\beta_{\infty, \infty}(V')     \hookrightarrow B^{2\beta-2}_{\infty, \infty}(V'),  
\end{equation} according to \cite[Theorem 4.6.1/2(i)]{RuSi}, establishes \eqref{v-mol-conv}.
\medskip
 
To show \eqref{equiv-dist}, for all $\varphi \in C^\infty_c(U')$, we choose $V' \Subset  U'$ such that ${\rm supp} \, \varphi \Subset V'$.  Applying \eqref{v-mol-conv}  we obtain, as required,  
$$
\begin{aligned}
 \tilde \kappa_{g'} [\varphi] & =  \lim_{\e\to 0} \int_{V'}(1 + |\nabla v_\e|^2)^{-2} (\Det {\mathcal D}^2 v_\e)\varphi  \\ & =  \lim_{\e\to 0}  \Det {\mathcal D}^2 v_\e [(1 + |\nabla v_\e|^2)^{-2} \varphi] \\ & = \Det {\mathcal D}^2 v [(1 + |\nabla v|^2)^{-2} \varphi],
\end{aligned}
$$ where the last equality is obtained by an extension argument via the duality \eqref{dlty}. \end{proof}
 
The following proposition reduces the problem to the case we studied in  Section \ref{MA-sec} regarding the very weak Monge-Amp\`ere equation:

\begin{proposition}\label{on-v}
$~$
\begin{itemize}
\item[(i)]  Assume that $\kappa_g\ge 0$ as a distribution.  Let  $p, x_0, U'$ as above and let  $V' \Subset U'$ be any open set containing $x_0$. Then there exists an open disk $B'\Subset V'$ containing $x_0$ such that   $\Det {\mathcal D}^2 v\ge 0$ in $B'$. 
\item[(ii)] If $\kappa_g \not \equiv 0$, then, there exists $x_0 \in \Omega$, and an open disk $B'$ containing $x_0$, such that $\Det {\mathcal D}^2 v\not \equiv 0$ in  $B'$. 
 
 \end{itemize}  
\end{proposition}

\begin{proof}
Let $v_\e$ be the mollifying sequence defined on $V'$  as above.  Consider the sequence of Riemannian metrics 
$$
g'_{(\e)} : = (\nabla \xi_{v_\e})^T  \nabla \xi_{v_\e} = {\rm Id} + \nabla v_\e\otimes \nabla v_\e, 
$$ for the graph parameterization $\xi_{v_\e} =(x, v_\e(x))$. It is a well known fact for smooth graphs that 
$$
\kappa_{g'_{(\e)}} = (1 + |\nabla v_\e|^2)^{-2} (\det \nabla^2 v_\e) = 
 (1 + |\nabla v_\e|^2)^{-2} (\Det {\mathcal D}^2 v_\e).
$$
By the estimates \eqref{2-commute}
 $$
 \|g'_\e\|_{1,V'}+  \|g'_{(\e)}\|_{1,V'} \le C \e^{\alpha-1},
 $$
 and
 $$
\|g'_\e - g'_{(\e)}\|_{j,V'} \le \|(\nabla v\otimes \nabla v)_\e -   \nabla v_\e\otimes \nabla v_\e\|_{j,V'} \le C \e^{2\alpha-j}.
$$ Now, fix the open disk $B' \ni x_0$  in $V'$ according to  Lemma \ref{asym-pull}. Note that 
$$
\det g'_{(\e)} = 1 + |\nabla v_\e|^2 \ge 1      
$$ and as  $\|\det g'_\e - (1+ |\nabla v|^2)\|_{0;V'}= \|\det g'_\e - \det g'\|_{0;V'} \le C\e^{\alpha}$, wee also have
\begin{equation}\label{bd-det-g'}
\det g'_\e \ge 1/2
\end{equation}  for $\e$ small enough.  In view of the uniform $C^0$-boundedness of both $g'_\e$ and $g'_{(\e)}$, and via an application of 
Lemma \ref{g-C1-conv} we have for all $\varphi \in C^\infty_c (B')$
 $$
 |\kappa_{g'_\e}[\varphi] - \kappa_{g'_{(\e)}}[\varphi] |  \le C \e^{3\alpha-2} \xrightarrow {\e\to 0} 0.
 $$ On the other hand, Lemmas  \ref{g-C1-conv} and \ref{asym-pull} and \eqref{bd-det-g'} together  imply  in the same manner
 $$
  |\kappa_{g'_\e}[\varphi] - \kappa_{\xi_\e^\ast g}[\varphi] | \le C \e^{3\alpha-2} \xrightarrow {\e\to 0} 0.
 $$ Combining the last two convergences, and taking into account \eqref{equiv-dist} and \eqref{v-mol-conv} we  finally obtain for all  $\varphi\in C^\infty_c(B')$:
 $$
 \begin{aligned}
 \Det {\mathcal D}^2 v  [(1 + |\nabla v|^2)^{-2} \varphi]  & =    \tilde \kappa_{g'}  [ \varphi] 
 \\& =  \lim_{\e \to 0} \int_{V'}(1 + |\nabla v_\e|^2)^{-2} (\Det {\mathcal D}^2 v_\e) \, \varphi  \\
 &= \lim_{\e \to 0}\kappa_{g'_{(\e)}}[\varphi] 
  = \lim_{\e \to 0} \kappa_{\xi_\e^\ast g}[\varphi] 
 \\& = \lim_{\e \to 0} (\kappa_{g} \circ \xi_\e )[\varphi]
  = (\kappa_g \circ \xi) [\varphi] 
 \\ & = \kappa_g  [|J(\xi)|^{-1} \varphi \circ \xi^{-1}],
\end{aligned}
 $$ where we used Lemma \ref{coord-ch}, Corollary \ref{compo-xi} in the last two lines. (i)  immediately follows by an approximation argument. 
 \medskip
 
  To see (ii), note that if $\kappa_g \not \equiv 0$, then there is a point $p = u(x_0) \in S = u(\Omega)$ for which for all $r>0$, there exists    $\varphi \in C^\infty_c(B_r(x_0))$ for which     $\kappa_g[\varphi] \neq 0$.  \end{proof} 

Applying Proposition \ref{convex2},  it follows immediately that under the assumptions of Theorem \ref{iso-bou-cur}, and for any $p\in S''$, 
$\xi_v(B') = u\circ \xi(B')$, which is an open neighborhood of $p$ in $S''$, is a surface of bounded extrinsic curvature, of nonnegative curvature, with the absolute curvature of $\xi_v(B')$ bounded e.g.\@ by the finite value $\||J(\xi)|^{-1} (1+|\nabla v|^2) ^2\|_{0;B'} \mu_{\kappa_g}(B')$. 
 
\medskip

Remember that $ \overline{S'}  \subset S''$ is compact. Therefore the local property on $S''$ is sufficient to show via  a covering argument and application of compactness property that $S'$ is a surface of bounded extrinsic curvature and of nonnegative curvature.  If the distributional   curvature of $g$ is nonzero, Proposition \ref{convex2} once again implies that for a suitable choice of $\Omega' \Subset \Omega$, and $p\in S'= u(\Omega')$, the positive curvature of  a neighborhood of $p$ in $S'$ is nonzero, and hence, since the absolute curvature is $\sigma$-additive nonnegative set function, the positive curvature of $S'$ is nonzero as required. The proof of Theorem \ref{iso-bou-cur} is complete. 
     \section{Proof of Theorem \ref{full-plane}}\label{sec-full-plane}
 
 By the assumptions of Theorem \ref{full-plane}, $v\in C^{1,\alpha}(\R^2)$  is a solution to $\Det{\mathcal D}^2 v = f\ge 0$ in $\mathcal{D'}(\R^2)$, where $2/3 < \alpha <1$ and $f\not \equiv 0$.  Since $\mu_f$ is nonzero, there exists $x\in \R^2$ for which $\mu_f(U)>0$ for some open disk $U$ containing $x$. Apply Proposition \ref{convex2} to increasing disks $\Omega= B_R(x)$ as $R\to\infty$,  and deduce that the graph of $v$ over $\R^2$ is a complete surface of nonnegative curvature with total nonzero curvature. \cite[Theorem 2, p.\@ 615]{Po73} will now imply that $S$ is an unbounded convex surface. As a consequence, $v$ is either globally convex or concave.   

\medskip

 It remains to be shown that $v$ is an Alexandrov solution to $\det\nabla^2v = \mu_f$ in $\R^2$.  We first claim that $\nabla  v$ is globally 1-1 on the set 
$\R^2_r$ of points $x\in \R^2$ for which $\xi_v(x)$ is regular.  This is obvious since the convexity of $v$ implies that if $\nabla v(x) = \nabla v(z)$ for $x\neq z$, then $\nabla v$ is constant on the segment $[x,z]$, so that $\xi_v(x)$ cannot be a regular point. 

\medskip

Now, consider any bounded piece-wise smooth open $U\subset \R^2$.  Note that by \cite[Lemma 4]{CDS}, $|\nabla v(\partial U)|=0$, a fact that we will repeatedly use in what follows. It is sufficient to prove that 
$\mu_v(U) = \mu_f(U)$ to conclude the proof.  Note that  by \cite[Theorems 1 and 4, p.\@ 590-591]{Po73}, we have 
${\mathcal H}^2(N(\R^2 \setminus \R^2_r)) = 0$, and thus  $|\nabla v (\R^2 \setminus \R^2_r)| = 0$.  Hence 
\begin{equation}\label{W1-1}
|U'|= |\nabla v(U)\setminus \nabla v(\partial U)|,
\end{equation} where  $U':= \nabla v(U)\setminus (\nabla v(\partial U) \cup \nabla v (\R^2 \setminus \R^2_r))$.  We observe that for all $y\in U'$, 
$(\nabla v)^{-1}(y) \subset \R^2_r$, which, since  $\nabla v$ is injective on $\R^2_r$,  yields that for any such $y$ that $(\nabla v)^{-1}(y) =\{x\}$ for some $x\in \R^2$. We now choose $r$ such that $O:=\xi_v(B_r(x))$  is as in \eqref{reg_def} and arguing as in Proposition \ref{convex2} we note that $\deg (\nabla v, B_{r/2}(x), y) =1$.  For   $y\in U'$, there is no other solution to $\nabla v(z) =y$ in $U$ other than $x$, which implies by the properties of the   \cite[Corollaire 2.4]{Kav}:
$$
\forall  y\in U' \qquad
\deg(\nabla v, U, y)   =\deg (\nabla v, B_{r/2}(x), y)= 1,$$
 and thus (\ref{area2}), \eqref{W1-1} and the fact that $|\nabla v(\partial U)|=0$  imply
$$\mu_v(U) = |\nabla v(U)|  = |U'|   = \int_{U'}    \deg(\nabla v, U, y)~\mbox{d}y 
  = \int_{\nabla v(U)\setminus\nabla v(\partial U)} \deg(\nabla v, U, y)~\mbox{d}y =  \mu_f(U).$$
 This concludes the proof of Theorem \ref{full-plane}.

\end{document}